\documentclass[12pt]{amsart}

\usepackage{multirow, amsthm, amsmath}

\newcommand{\Q}{\mathbb{Q}}

\newcommand{\Z}{\mathbb{Z}}

\newcommand{\MM}{\mathbb{M}}

\newcommand{\C}{\mathbb{C}}

\newcommand{\N}{\mathbb{N}}
\newcommand{\F}{\mathbb{F}}

\newcommand\GLtwoZ{\operatorname{GL}_2(\Z)}
\newcommand\GLtwon{\operatorname{GL}_{2n}}
\newcommand\SLtwoZ{\operatorname{SL}_2(\Z)}
\newcommand\SLtwoF{\operatorname{SL}_2(\F_2)}

\newcommand\Spn{\operatorname{Sp}_n}

\newcommand\SptwoZ{\operatorname{Sp}_2(\Z)}
\newcommand\SptwoQ{\operatorname{Sp}_2(\Q)}
\newcommand\SptwoF{\operatorname{Sp}_2(\F_2)}

\newcommand\Half{{\mathcal H}}

\newcommand\inv{^{-1}}

\newcommand\ep{\epsilon}
\newcommand\s{\sigma}
\newcommand\Jac{\Gamma_{\infty}(\Z)}
\newcommand\SSS{\mathcal S}
\newcommand\ldL{\llcorner}
\newcommand\rdL{\lrcorner}
\newcommand\luL{\ulcorner}
\newcommand\ruL{\urcorner}

\newcommand{\diag}{\text{diag}}
\newcommand{\tr}{\text{tr}}

\newcommand\FJ{\operatorname{FJ}}
\newcommand\ord{\operatorname{ord}}
\newcommand\Grit{\operatorname{Grit}}

\newcommand\smtwomat[4]{
{\bigl(
\genfrac{}{}{0pt}{1}{#1}{#3}\,
\genfrac{}{}{0pt}{1}{#2}{#4}
\bigr)}}\newcommand{\Xtwo}{\mathcal{X}_2}
\newcommand{\KN}{K(N)}
\newcommand{\pX}{\mathcal{X}_2(p)}

\newcommand{\NXs}{ \mathcal{X}_2^{\text{semi}}(N)}
\newcommand{\NXss}{ \mathcal{X}_2^{\text{\rm semi}}(N)}
\newcommand{\Gt}{\tilde \Gamma(t)}
\newcommand{\Gttwo}{\tilde \Gamma(2)}

\newcommand{\zzz}{ }

\newcommand{\thmref}[1]{Theorem~\ref{#1}}
\newcommand{\lemref}[1]{Lemma~\ref{#1}}

\newtheorem{theorem}{Theorem}[section]

\newtheorem{lemma}[theorem]{Lemma}

\newtheorem{proposition}[theorem]{Proposition}

\newtheorem{corollary}[theorem]{Corollary}

\newtheorem{definition}[theorem]{Definition}%[section] 
   \def\tableline{\hbox to \hsize}
\newbox\hdbox%
\newcount\hdrows%
\newcount\multispancount%
\newcount\ncase%
\newcount\ncols% This is the number of primary text columns in the table.
\newcount\nrows%
\newcount\nspan%
\newcount\ntemp%
\newdimen\hdsize%
\newdimen\newhdsize%
\newdimen\parasize%
\newdimen\spreadwidth%
\newdimen\thicksize%
\newdimen\thinsize%
\newdimen\tablewidth%
\newif\ifcentertables%
\newif\ifendsize%
\newif\iffirstrow%
\newif\iftableinfo%
\newtoks\dbt%
\newtoks\hdtks%
\newtoks\savetks%
\newtoks\tableLETtokens%
\newtoks\tabletokens%
\newtoks\widthspec%
\tableinfotrue%
\catcode`\@=11%  Allows use of "@" in macro names, like PLAIN.TEX does.
\def\tstrut{\vrule height3.1ex depth1.2ex width0pt}%
\def\and{\char`\&}%  Allows us to get an `&' in the text.  This is the
\def\tablerule{\noalign{\hrule height\thinsize depth0pt}}%
\def\thickrule{\noalign{\hrule height\thicksize depth0pt}}%
\def\ctr#1{\hfil\ #1\hfil}%
\tablewidth=-\maxdimen%
\spreadwidth=-\maxdimen%
\def\tabskipglue{0pt plus 1fil minus 1fil}%
\centertablestrue%
\gdef\ARGS{########}%  Produces the correct number of #'s in the preamble
\gdef\headerARGS{####}%  Same as \ARGS, but used in \header macros.
\def\@mpersand{&}%  Allows us to get alignment tab characters later
{\catcode`\|=13%  Make |'s locally active.
\gdef\letbarzero{\let|0}%  Globally define a macro that allows us to
\gdef\letbartab{\def|{&&}}%
\gdef\letvbbar{\let\vb|}%
}%  End of locally active |'s.
{\catcode`\&=4%  Make these alignment tabs.
\def\ampskip{&\omit\hfil&}%  This local macro skips a vertical rule.
\catcode`\&=13%  Now make &'s into active macros.
\let&0%  This allows us to expand \ampskip in the next \xdef without
\xdef\letampskip{\def&{\ampskip}}%
\gdef\letnovbamp{\let\novb&\let\tab&}
}%  End of locally active &'s.
\def\begintable{%  Here we make |'s and &'s active characters so we can
   \begingroup%
   \catcode`\|=13\letbartab\letvbbar%
   \catcode`\&=13\letampskip\letnovbamp%
   \def\multispan##1{%  We must redefine \multispan to count the number
      \omit \mscount##1%
      \multiply\mscount\tw@\advance\mscount\m@ne%
      \loop\ifnum\mscount>\@ne \sp@n\repeat%
   }%  End of \multispan macro.
   \def\|{%
      &\omit\widevline&%
   }%
   \ruledtable%  Now we call \ruledtable to do the real work.
}%  End of \begintable macro.
\long\def\ruledtable#1\endtable{%
   \offinterlineskip%  Needed to make rules touch each other.
   \tabskip 0pt%  Needed for same reason as \offinterlineskip.
   \def\widevline{\vrule width\thicksize}%  Make outer \vrule's wider.
   \def\endrow{\@mpersand\omit\hfil\crnorm\@mpersand}%
   \def\crthick{\@mpersand\crnorm\thickrule\@mpersand}%
   \def\crnorule{\@mpersand\crnorm\@mpersand}%
   \let\nr=\crnorule%  A shorter abbreviation.
   \def\endtable{\@mpersand\crnorm\thickrule}%
   \let\crnorm=\cr%  Allows us to use \cr for our own purposes.
   \edef\cr{\@mpersand\crnorm\tablerule\@mpersand}%
   \the\tableLETtokens%  Get the user's extra \let's, if any.
   \tabletokens={&#1}%  We add an extra alignment tab to the beginning
   \countROWS\tabletokens\into\nrows%
   \countCOLS\tabletokens\into\ncols%
   \advance\ncols by -1%
   \divide\ncols by 2%
   \advance\nrows by 1%
   \iftableinfo %
      \immediate\write16{[Nrows=\the\nrows, Ncols=\the\ncols]}%
   \fi%
   \ifcentertables
      \ifhmode \par\fi%  Make sure we are in vertical mode.
      \tableline{%  The final table comes out as an \hbox of width the \hsize.
      \hss%  The final table will be centered left-to-right.
   \else %
      \hbox{%
   \fi
      \vbox{%
         \makePREAMBLE{\the\ncols}%  Generate the preamble.
         \edef\next{\preamble}%  This line and the next line force the
         \let\preamble=\next%    expansion of all \ARGS tokens into the
         \makeTABLE{\preamble}{\tabletokens}%  Go do the \halign here.
      }%  End of \vbox.
      \ifcentertables \hss}\else }\fi%  Finish the centering effect.
   \endgroup%  Return all local macros and parameters to their outside
   \tablewidth=-\maxdimen%  Reset \tablewidth to normal.
   \spreadwidth=-\maxdimen% Same for \spreadwidth.
}%  End of macro \ruledtable.
\def\makeTABLE#1#2{%  Does an \halign for the \ruledtable macro.
   {%  Start of local parameter values.
   \let\ifmath0%     These macros would cause trouble if they were to be
   \let\header0%     expanded in the following \xdef; we \let them be
   \let\multispan0%  equal to a digit, because digits can't be expanded.
   \ncase=0%
   \ifdim\tablewidth>-\maxdimen \ncase=1\fi%
   \ifdim\spreadwidth>-\maxdimen \ncase=2\fi%
   \relax%  This \relax is absolutely necessary, without it the following
   \ifcase\ncase %
      \widthspec={}%
   \or %
      \widthspec=\expandafter{\expandafter t\expandafter o%
                 \the\tablewidth}%
   \else %
      \widthspec=\expandafter{\expandafter s\expandafter p\expandafter r%
                 \expandafter e\expandafter a\expandafter d%
                 \the\spreadwidth}%
   \fi %
   \xdef\next{%  We must force the preamble to be expanded BEFORE the
      \halign\the\widthspec{%
      #1%  This is the preamble text.
      \noalign{\hrule height\thicksize depth0pt}%  Makes the top \hrule.
      \the#2\endtable%  This is the main body.
      }%  End of \halign.
   }%  End of \next.
   }%  End of local values.
   \next%  This \next must be outside of the local values, because now
}%  End of macro \makeTABLE.
\def\makePREAMBLE#1{%  This macro generates the necessary preamble for a
   \ncols=#1%  Get the number of columns desired.
   \begingroup%  Start local parameter definitions.
   \let\ARGS=0%  This is the key to the whole thing; it prevents \ARGS
   \edef\xtp{\widevline\ARGS\tabskip\tabskipglue%
   &\ctr{\ARGS}\tstrut}%  A 1-column preamble.  Gets the sizing right.
   \advance\ncols by -1%  One column has been generated; decrement the
   \loop%  Append as many further columns as needed to the preamble.
      \ifnum\ncols>0 %
      \advance\ncols by -1%
      \edef\xtp{\xtp&\vrule width\thinsize\ARGS&\ctr{\ARGS}}%
   \repeat
   \xdef\preamble{\xtp&\widevline\ARGS\tabskip0pt%
   \crnorm}%  Adds the last \vrule.
   \endgroup%  End of local parameters.
}%  End of macro \makePREAMBLE.
\def\countROWS#1\into#2{%  This counts the number of rows in #1 by
   \let\countREGISTER=#2%
   \countREGISTER=0%
   \expandafter\ROWcount\the#1\endcount%
}%
\def\ROWcount{%
   \afterassignment\subROWcount\let\next= %
}%
\def\subROWcount{%
   \ifx\next\endcount %
      \let\next=\relax%
   \else%
      \ncase=0%
      \ifx\next\cr %
         \global\advance\countREGISTER by 1%
         \ncase=0%
      \fi%
      \ifx\next\endrow %
         \global\advance\countREGISTER by 1%
         \ncase=0%
      \fi%
      \ifx\next\crthick %
         \global\advance\countREGISTER by 1%
         \ncase=0%
      \fi%
      \ifx\next\crnorule %
         \global\advance\countREGISTER by 1%
         \ncase=0%
      \fi%
      \ifx\next\header %
         \ncase=1%
      \fi%
      \relax%
      \ifcase\ncase %
         \let\next\ROWcount%
      \or %
         \let\next\argROWskip%
      \else %
      \fi%
   \fi%
   \next%
}%  End of macro \subROWcount.
\def\counthdROWS#1\into#2{%
\dvr{10}%
   \let\countREGISTER=#2%
   \countREGISTER=0%
\dvr{11}%
\dvr{13}%
   \expandafter\hdROWcount\the#1\endcount%
\dvr{12}%
}%
\def\hdROWcount{%
   \afterassignment\subhdROWcount\let\next= %
}%
\def\subhdROWcount{%
   \ifx\next\endcount %
      \let\next=\relax%
   \else%
      \ncase=0%
      \ifx\next\cr %
         \global\advance\countREGISTER by 1%
         \ncase=0%
      \fi%
      \ifx\next\endrow %
         \global\advance\countREGISTER by 1%
         \ncase=0%
      \fi%
      \ifx\next\crthick %
         \global\advance\countREGISTER by 1%
         \ncase=0%
      \fi%
      \ifx\next\crnorule %
         \global\advance\countREGISTER by 1%
         \ncase=0%
      \fi%
      \ifx\next\header %
         \ncase=1%
      \fi%
\relax%
      \ifcase\ncase %
         \let\next\hdROWcount%
      \or%
         \let\next\arghdROWskip%
      \else %
      \fi%
   \fi%
   \next%
}%
{\catcode`\|=13\letbartab
\gdef\countCOLS#1\into#2{%
   \let\countREGISTER=#2%
   \global\countREGISTER=0%
   \global\multispancount=0%
   \global\firstrowtrue
   \expandafter\COLcount\the#1\endcount%
   \global\advance\countREGISTER by 3%
   \global\advance\countREGISTER by -\multispancount
}%
\gdef\COLcount{%
   \afterassignment\subCOLcount\let\next= %
}%
{\catcode`\&=13%
\gdef\subCOLcount{%
   \ifx\next\endcount %
      \let\next=\relax%
   \else%
      \ncase=0%
      \iffirstrow
         \ifx\next& %
            \global\advance\countREGISTER by 2%
            \ncase=0%
         \fi%
         \ifx\next\span %
            \global\advance\countREGISTER by 1%
            \ncase=0%
         \fi%
         \ifx\next| %
            \global\advance\countREGISTER by 2%
            \ncase=0%
         \fi
         \ifx\next\|
            \global\advance\countREGISTER by 2%
            \ncase=0%
         \fi
         \ifx\next\multispan
            \ncase=1%
            \global\advance\multispancount by 1%
         \fi
         \ifx\next\header
            \ncase=2%
         \fi
         \ifx\next\cr       \global\firstrowfalse \fi
         \ifx\next\endrow   \global\firstrowfalse \fi
         \ifx\next\crthick  \global\firstrowfalse \fi
         \ifx\next\crnorule \global\firstrowfalse \fi
      \fi%  End of \iffirstrow.
\relax%\out{subCOL-->  ncase=O\the\ncaseE}
      \ifcase\ncase %
         \let\next\COLcount%
      \or %
         \let\next\spancount%
      \or %
         \let\next\argCOLskip%
      \else %
      \fi %
   \fi%
   \next%
}%
\gdef\argROWskip#1{%
   \let\next\ROWcount \next%
}%  End of macro \argskip.
\gdef\arghdROWskip#1{%
   \let\next\ROWcount \next%
}%  End of macro \arghdROWskip.
\gdef\argCOLskip#1{%
   \let\next\COLcount \next%
}%  End of macro \argskip.
}%  End of active &'s.
}%  End of active |'s.
\def\spancount#1{%\out{spancount--->\meaning#1}
   \nspan=#1\multiply\nspan by 2\advance\nspan by -1%
   \global\advance \countREGISTER by \nspan
   \let\next\COLcount \next}%
\def\dvr#1{\relax}%
\def\header#1{%
\dvr{1}{\let\cr=\@mpersand%
\hdtks={#1}%
\counthdROWS\hdtks\into\hdrows%
\advance\hdrows by 1%
\ifnum\hdrows=0 \hdrows=1 \fi%
\dvr{5}\makehdPREAMBLE{\the\hdrows}%
\dvr{6}\getHDdimen{#1}%
{\parindent=0pt\hsize=\hdsize{\let\ifmath0%
\xdef\next{\valign{\headerpreamble #1\crnorm}}}\dvr{7}\next\dvr{8}%
}%
}\dvr{2}}%  End of macro \header.
\def\makehdPREAMBLE#1{%This macro generates the necessary preamble for a
\dvr{3}%
\hdrows=#1%  Get the number of columns desired.
{%  Start local parameter definitions.
\let\headerARGS=0%
\let\cr=\crnorm%
\edef\xtp{\vfil\hfil\hbox{\headerARGS}\hfil\vfil}%
\advance\hdrows by -1%  One row has been generated; decrement the
\loop%  Append as many further rows as needed to the preamble.
\ifnum\hdrows>0%
\advance\hdrows by -1%
\edef\xtp{\xtp&\vfil\hfil\hbox{\headerARGS}\hfil\vfil}%
\repeat%
\xdef\headerpreamble{\xtp\crcr}%
}%  End of local parameters.
\dvr{4}}%  End of \makehdPREAMBLE.
\def\getHDdimen#1{%
\hdsize=0pt%
\getsize#1\cr\end\cr%
}%  End of macro getHDdimen.
\def\getsize#1\cr{%
\endsizefalse\savetks={#1}%
\expandafter\lookend\the\savetks\cr%
\relax \ifendsize \let\next\relax \else%
\setbox\hdbox=\hbox{#1}\newhdsize=1.0\wd\hdbox%
\ifdim\newhdsize>\hdsize \hdsize=\newhdsize \fi%
\let\next\getsize \fi%
\next%
}%
\def\lookend{\afterassignment\sublookend\let\looknext= }%
\def\sublookend{\relax%
\ifx\looknext\cr %
\let\looknext\relax \else %
   \relax
   \ifx\looknext\end \global\endsizetrue \fi%
   \let\looknext=\lookend%
    \fi \looknext%
}%
\def\tablelet#1{%
   \tableLETtokens=\expandafter{\the\tableLETtokens #1}%
}%
\catcode`\@=12%  Change @'s back to their normal category code.
\begin{document}

\title[Fourier~Jacobi Expansions]{Jacobi forms that characterize paramodular forms}

\author{Tomoyoshi Ibukiyama}
\address{Department of Mathematics, Graduate School of Mathematics \\ 
Osaka University, Machikaneyama 1-1, Toyonaka, Osaka, 560-0043 Japan}
\email{ibukiyam@math.sci.osaka-u.ac.jp}

\author{Cris Poor}
\address{Department of Mathematics, Fordham University, Bronx, NY 10458}
\email{poor@fordham.edu}

\author{David S. Yuen}
\address{Department of Mathematics and Computer Science, Lake Forest College, 555 N. Sheridan Rd., Lake Forest, IL 60045}
\email{yuen@lakeforest.edu}

%\dedicatory{This paper is dedicated to our fathers.}

\subjclass[2000]{Primary 11F46, 11F50}

\date{ \today.}

\keywords{Jacobi forms, Paramodular forms}

\begin{abstract}
 The Fourier Jacobi expansions of paramodular forms are characterized 
 from among all formal series of Jacobi forms by two conditions on 
 the Fourier coefficients of the Jacobi forms: a growth condition and a 
 set of linear relations.  Examples, both theoretical and computational,  
 indicate that the growth condition may be superfluous.  
\end{abstract}
\maketitle
%\centerline{Draft: \today}

%\rightheadtext{Jacobi forms}

%Added by Ibukiyama: My preference for change.\\
%(1) Add some explanation before the table to show  
%$\dim S_{k}(\Gamma[5])<\sum_{j}\dim J_{k,5j}^{cusp}(j)$.
%\\
%(2) For the notation $\Gamma[N]$, I prefer $K(N)$ or $\Gamma^{para}(N)$. 

\section{Introduction}\label{sec: NewIntro}

For theoretical purposes it would be nice 
to characterize the Fourier Jacobi expansions of Siegel paramodular forms of degree two 
from among all formal power series with Jacobi forms as coefficients.  
For computational purposes it would be nice if the 
characterization were in terms of linear relations among the Fourier coefficients of
the various Jacobi forms.  We achieve this goal only in a few cases.

The linear relations we study arise from a symmetry possessed by the 
Fourier Jacobi expansions of paramodular forms.  Let $J_{k,m}$ denote the 
complex vector space of Jacobi forms of weight~$k$ and index~$m$.  
Let $\Gamma$ be a group commensurable with $\SptwoZ$ and denote by 
$M_k\left( \Gamma \right)$ the complex vector space of Siegel modular forms of weight~$k$ 
automorphic with respect to~$\Gamma$.  One commensurable family is given 
by the paramodular groups~$\KN $: 
$$
\KN =
\begin{pmatrix} 
*  &  N*  &  *   &  * \\
*& *  &  *  &  */N  \\
*& N*  &  *  &  *  \\
N*& N*  &  N*  &  *  
\end{pmatrix} 
\cap \SptwoQ, 
\text{ where $*\in \Z$.}
$$
Each paramodular form $f \in M_k\left(\KN \right)$ 
has a Fourier Jacobi expansion 
$ f\smtwomat{\tau}{z}{z}{\omega}= 
\sum_{m \ge 0: N \vert m} \phi_m(\tau,z)\,e(m\omega)$ 
where $\smtwomat{\tau}{z}{z}{\omega}$ is in the Siegel upper half space 
and $\phi_m \in J_{k,m}$.  
These Jacobi forms~$\phi_m$ are not independent and possess a symmetry 
that is best expressed by using a normalizer~$\mu_N$ of the 
paramodular group~$\KN $  satisfying $\mu_N^2=-I_4$ and 
given by $\mu_N =\smtwomat{-F_N'}{0}{0}{F_N}$, where the Fricke involution 
$F_N=\frac1{\sqrt{N}}\smtwomat{0}{1}{-N}{0}$ is the usual normalizer of 
$\Gamma_0(N)=\{\smtwomat{a}{b}{c}{d}\in \SLtwoZ: N \vert c \}$.

For $\ep = \pm 1$, let 
$M_k(\KN )^{\ep}= \{ f \in M_k(\KN ): f \vert_k \mu_N = \ep f \}$ 
be the plus and minus eigenspaces of~$\mu_N$.  
Let the Fourier Jacobi expansion map, 
$\FJ: M_k\left(\KN \right)^{\ep} \to \prod_{m\in\Z : \,m \ge 0, N \vert m}J_{k,m}$,  be 
 defined by 
$\FJ(f)=\sum_{m: N \vert m} \phi_m \xi^m$ and write, for $(\tau, z) \in \Half_1 \times\C$,  
\begin{equation*}
\phi_m(\tau,z)= 
\sum_{n,r\in\Z:\,\, 4mn \ge r^2, \,\,\,n \ge 0} 
c(n,r;\phi_m)\,e(n\tau+rz).  
\end{equation*}
These coefficients possess the symmetry 
\begin{equation}
\label{eq22} 
c(n,r;\phi_m)= \ep c(m/N,-r;\phi_{nN}).  
\end{equation}

%%%%%
%Added by T. Ibukiyama:
%By the way, $f\in M_k(\Gamma[N])^{\epsilon}$ is a cusp form if and only if 
%$FJ(f)\in \prod_{m\in \Z;m\ge q0,N|m}J_{k,m}^{cusp}$. This is proved by the fact
%that one-dimensional cusps are represented by matrices of the shape $\begin{pmatrix}
%A & 0 \\ 0 & D \end{pmatrix}$ (cf. \cite{PYcusp}).  

%Cris modified above comment.  
We mention that $f\in M_k(\KN)^{\epsilon}$ is a cusp form if and only if 
$\FJ(f)\in \prod_{m\in \Z;m\ge 0,N|m}J_{k,m}^{\rm cusp}$. 
This nontrivial assertion follows from the representation of the 
one-dimensional cusps by matrices of the shape 
$\smtwomat{A }{ 0}{0 }{ D}$.   In fact, the one-dimensional cusps correspond to 
divisors~$t$ of $N$ via $D^*=A=\smtwomat{1}{t}{0}{1}$, 
see  Reefschl\"ager \cite{Reef} or compare \cite{PYcusp}.

In Theorem~{2.2} we show that certain {\sl convergent\/} series of Jacobi forms 
satisfying the symmetry~{(\ref{eq22})} are in fact the Fourier Jacobi expansion 
of some Siegel paramodular form.  However, the real question motivating this 
article is: Are {\sl formal\/} series of Jacobi forms 
satisfying the symmetry~{(\ref{eq22})}  the Fourier Jacobi expansions 
of Siegel paramodular forms?  
Work of H. Aoki \cite{Aoki} essentially answers this question affirmatively for $N=1$ 
and we prove this for $N\in\{2,3,4\}$  as well by following his method.  
Let us give a more definite formulation.  
\begin{definition} 
\label{def3} 
Let $\NXs=\{ \smtwomat{a}{b}{b}{c} \ge 0: a,2b,c \in \Z \text{ and } N \vert c \}$ for  $N \in \N$.  
For $k \in \Z$, 
let  $\Phi = \sum_{m: N \vert m} \phi_m \xi^m \in  \prod_{m\ge 0: N \vert m}J_{k,m}$ 
be a formal power series whose coefficients are Jacobi forms.  
For $\ep \in \{-1,1\}$, we say that $\Phi$ satisfies the {\rm Involution$(\ep)$ condition\/}  if 
 $$
  \forall \smtwomat{n}{r/2}{r/2}{m} \in \NXss,  \,
c(n,r;\phi_m)= \ep\, c(\frac{m}{N}, -r; \phi_{nN}).
$$
We say that $\Phi$ satisfies the {\rm growth condition\/} if 
$$
  \forall \rho >1, \exists A >0: \forall \smtwomat{n}{r/2}{r/2}{m} \in \NXss, \,
\vert c(n,r; \phi_m) \vert \le A \rho^{n+m}. 
$$
Set 
$\MM_k(N)  ^{\ep}= 
\{ \Phi  \in  \prod_{m\ge 0: N \vert m}J_{k,m}: 
\text{ $\Phi$  satisfies  Involution$(\ep)$}  \}$.  
\end{definition}
We would like to know when the map 
$\FJ: M_k\left(\KN \right)^{\ep} \to  \MM_k(N) ^{\ep} $ 
is surjective.  In Theorem~{2.2} we show that this map 
surjects onto the subspace of $ \MM_k(N)  ^{\ep} $ that satisfies the 
growth condition, thereby giving at least one theoretical characterization 
of the Fourier Jacobi expansions of Siegel paramodular forms.  
Details aside, this amounts to the fact that the 
paramodular groups are generated by the Jacobi group and an involution.  
By following Aoki's method however, 
we do prove  the surjectivity of $\FJ$ onto $ \MM_k(N)  ^{\ep} $ for $N \le 4$.  
\begin{theorem}\label{Second}
Let $N \in \{1,2,3,4\}$ and $\epsilon \in \{-1,1\}$.  
For all weights $ k \in \Z$, 
the Fourier Jacobi expansion map~$\FJ$  
from paramodular 
forms to formal series of Jacobi forms that satisfy the Involution$(\ep)$ 
condition, $\FJ: M_k\left( \KN  \right)^{\ep} \to \MM_k(N)  ^{\ep}$, 
is an isomorphism.
\end{theorem}
%
%As a corollary we obtain new results for the generating functions of the 
%plus and minus eigenspaces, the full generating functions being known 
%for $N=2$ by T. Ibukiyama and F. Onodera \cite{IbukOnodera} and 
%for $N=3$ by T. Dern \cite{Dern}.  Our proofs use their results.  
%The generating function $\sum \dim M_k(\Gamma[4]) t^k$ 
%is given here for the first time by relying on the definitive results 
%of Igusa \cite{Igusa} for subgroups of $\Gamma_2$ that contain the principal subgroup $\Gamma_2(2)$.  
%These new results are:  

%%%%%wott
%Added by Ibukiyama as an alternative of the above:
% Edited somewhat by Cris
As a corollary we obtain new results for the generating functions of the 
plus and minus eigenspaces. 
For any prime $p$, $\dim S_{k}(K(p))$ is known in \cite{Ibukp} for $k>4$,  in \cite{Ibuk} 
for $k=3$, $4$, and for $p<349$ and $k=2$ in \cite{Para}. We can easily  show that the generalized 
Siegel $\Phi$ operator, the projection from $M_{k}(\KN)$ to the boundary of the 
Satake compactification, is always surjective for any $k$ for squarefree $N$.
Indeed, this is due to Satake \cite{Satake} when $k>4$, and, again for squarefree~$N$, the image is zero dimensional for $k=2$ and  
at most one dimensional for $k=4$  due to the known cusp configuration in \cite{Ibukcusp} for prime level and in 
\cite{PYcusp} for general $N$; furthermore, 
the lift of the Jacobi Eisenstein series of $J_{4,N}$ surjects to the image of  $\Phi$ when $k=4$.
So the generating function for $\dim M_{k}(K(p))$ 
can be easily given for any $p$ as long as we know $\dim S_{2}(K(p))$.  
In fact, the full generating functions are known for $N=2$ by T. Ibukiyama and F. Onodera 
\cite{IbukOnodera}, the plus and minus eigenspaces being given there also,  
and for $N=3$ by T. Dern \cite{Dern}. 
Our proofs use their results.
The generating function $\sum \dim M_k(K(4)) t^k$ 
is given here for the first time by relying on the definitive results 
of Igusa \cite{Igusa} for subgroups of $\Gamma_2$ that contain the principal subgroup $\Gamma_2(2)$.  
These new results are:  
%%%%%%%%%%%%%%%%%%%

\begin{align*}
\sum_{k \in \Z} \dim  M_k\left(K(2)\right)^{+}\,t^k &= 
\frac{1+t^{10}+t^{23}+t^{33} }{(1-t^4)(1-t^6)(1-t^8)(1-t^{12})},   \\
\sum_{k \in \Z} \dim  M_k\left(K(2)\right)^{-}\,t^k &= 
\frac{t^{11}+t^{12}+t^{21}+t^{22} }{(1-t^4)(1-t^6)(1-t^8)(1-t^{12})},   \\
\sum_{k \in \Z} \dim  M_k\left(K(3) \right)^{+}\,t^k &= 
\frac{1+t^{8}+t^{10}+t^{21}+t^{23}+t^{31} }{(1-t^4)(1-t^6)^2(1-t^{12})},   \\
\sum_{k \in \Z} \dim  M_k\left(K(3)  \right)^{-}\,t^k &= 
\frac{t^{9}+t^{11}+t^{12}+t^{19}+t^{20}+t^{22} }{(1-t^4)(1-t^6)^2(1-t^{12})}, \\    
\sum_{k \in \Z} \dim  M_k\left(K(4) \right)^{+}\,t^k &= 
\frac{1+t^6+t^{8}+t^{10}+t^{19}+t^{21}+t^{23}+t^{29} }{(1-t^4)^2(1-t^6)(1-t^{12})},   \\
\sum_{k \in \Z} \dim  M_k\left(K(4)  \right)^{-}\,t^k &= 
\frac{t^7+t^{9}+t^{11}+t^{12}+t^{17}+t^{18}+t^{20}+t^{22} }{(1-t^4)^2(1-t^6)(1-t^{12})}.  
\end{align*}

The question of the surjectivity of $\FJ:  M_k\left(\KN \right)^{\ep} \to \MM_k(N)  ^{\ep} $ 
is not idle and has applications to the computation of paramodular forms.  
To illustrate this, in section~4 we use the symmetry condition to compute 
$S_4\left( K(31) \right)^{\pm}$.  
These computations at least make it plausible that the growth condition is superfluous.  
Here one may also find a lemma showing that, for prime~$p$, 
initial Fourier Jacobi expansions 
$\pi_{pJ}\circ \FJ:  S_k(K(p))^{\ep} \to \prod_{j=1}^{J} J_{k, pj}^{\text{\rm cusp}}$ 
inject for $J \ge \ldL \dfrac{k}{10}\left( \dfrac{p^2+1}{p+1} \right) \rdL$.  

We thank Nils Skoruppa for his explanations to us about theta blocks.  
We thank Armand Brumer for suggesting that Fourier Jacobi expansions be used 
to compute spaces of paramodular cusp forms.

\section{A Characterization of Fourier Jacobi Expansions}\label{sec: Introduction}

For a ring $R$,  let 
$\Spn(R)=\{\sigma \in \GLtwon(R): \s'J\s=J\}$ define 
the symplectic group over $R$, where $J=\smtwomat{0}{I_n}{-I_n}{0}$ and 
$\s'$ is the transpose of $\s$.  The paramodular group~$\KN $,  
defined in the Introduction,  
is generated by the translations 
$\smtwomat{I}{S}{0}{I}$ with $S=\smtwomat{\alpha}{\beta}{\beta}{\gamma/N}$  
for $\alpha$, $\beta$, $\gamma \in \Z$, and the element $J(N)$, see \cite{Delzeith}, Theorem~9, 
$$
J(N)= 
\begin{pmatrix} 
0  & 0  &  1  &  0 \\
0 & 0  &  0  &  1/N  \\
-1 & 0  &  0  &  0  \\
0 & -N  & 0  &  0  
\end{pmatrix} .  
$$

Let $\Half_n$ denote the Siegel upper half space.  
For $k \in \Z$, the paramodular forms of weight~$k$, 
denoted by $M_k(\KN )$, are the $\C$-vector space of holomorphic 
$f:\Half_2\to\C$ with the property that $f \vert_k \s=f$ for all $ \s\in \KN $.  
The subspace of cusp forms is given by $S_k(\KN )=\{f \in M_k(\KN ): 
\forall \s \in \SptwoZ, \Phi(f \vert_k \s)=0 \}$.  
Here the slash action, 
$\left( f \vert_k \smtwomat{A}{B}{C}{D} \right)(\Omega)= 
\det(C\Omega+D)^{-k}\,f\left( (A\Omega+B)(C\Omega+D)\inv \right)$ 
and the $\Phi$ operator, 
$(\Phi f)(\tau)=\lim_{\lambda \to +\infty}f\smtwomat{i\lambda}{0}{0}{\tau}$, 
are the usual ones, see \cite{Fr}.  
Since $\mu_N^2$ acts trivially on modular forms, we may decompose paramodular forms 
into plus and minus forms:  
$M_k(\KN )= M_k(\KN )^{+} \oplus M_k(\KN )^{-}$ 
where $M_k(\KN )^{\ep}= \{ f \in M_k(\KN ): f \vert \mu = \ep f \}$ 
for $ \ep \in \{-1, 1\}$.  

Every paramodular form $f \in M_k(\KN )$ has a Fourier expansion 
$$
f(\Omega)=\sum_{   T \in \NXs } a(T;f)\, e\left( \langle \Omega, T \rangle \right)
$$
supported on $\NXs=\{ \smtwomat{a}{b}{b}{c} \ge 0: a,2b,c \in \Z \text{ and } N \vert c \}$;  
here $e(z)=e^{2 \pi i z}$ and $\langle A, B \rangle =\tr(AB)$.  
Setting $T[\s]=\s' T \s$, we additionally have 
$a( T[\s];f)=\det(\s)^k a(T;f)$ for all 
$\s \in {\hat \Gamma}^0(N)=\{ \smtwomat{a}{b}{c}{d} \in\GLtwoZ: N \vert b \}$.  
Note that the action of ${\hat \Gamma}^0(N)$ stabilizes $\NXs$.  
If we write $\Omega=\smtwomat{\tau}{z}{z}{\omega}\in\Half_2$ 
and collect the Fourier expansion of $f$ in powers of $\xi=e(\omega)$, 
then we obtain the Fourier Jacobi expansion of $f$: 
$ f\smtwomat{\tau}{z}{z}{\omega}= 
\sum_{m \ge 0: N \vert m} \phi_m(\tau,z)\,\xi^m$ where the 
\begin{equation}
\label{eq1} 
\phi_m(\tau,z)= 
\sum_{n,r\in\Z: \smtwomat{n}{r/2}{r/2}{m} \ge 0, \,\,\,n \ge 0} 
a\left( \smtwomat{n}{r/2}{r/2}{m} ;f \right) e(n\tau) e(rz)
\end{equation}
are Jacobi forms of weight~$k$ and index~$m$.  
This Fourier Jacobi expansion is term by term invariant under the group,  
$$
\Jac=
\begin{pmatrix} 
*  &  0  &  *  &  * \\
* & *  &  *  &  *  \\
* & 0 &  *  &  *  \\
0 & 0  &  0  &  *  
\end{pmatrix} 
\cap \SptwoZ, 
$$
and this is one motivation for the definition of Jacobi forms.  
\begin{definition} 
Let $k,m \in \Z_{\ge 0}$.  
The $\C$-vector space $J_{k,m}$ of Jacobi forms of weight~$k$ and index~$m$ 
is the set of holomorphic $\phi:\Half_1 \times \C \to \C$ satisfying: \newline 
{\rm 1.)  \/} 
$\forall\,\s\in\Jac, {\tilde \phi}\vert_k \s= {\tilde \phi}$, where 
${\tilde \phi}:\Half_2\to \C$ is defined by 
${\tilde \phi}\smtwomat{\tau}{z}{z}{\omega}=\phi(\tau,z)e(m\omega)$.  \newline 
{\rm 2.)\/} Setting $q=e(\tau)$ and $\zeta=e(z)$, the 
 Fourier series of $\phi$ has the form:   
$\phi(\tau,z)= \sum_{ n,r\in \Z: n \ge 0,\,\, 4mn \ge r^2  }c(n,r;\phi) q^n\zeta^r$.  
\end{definition}
The vector space of Jacobi cusp forms $J_{k,m}^{\text{cusp}}$ is defined by replacing 
$4mn \ge r^2$ by $4mn   > r^2$ in item~$2$.  
If we identify a sequence $(\phi_m) \in \prod_{m\in\Z: \,m \ge 0, N \vert m}J_{k,m}$
with the formal power series $\sum_{m: N \vert m} \phi_m \xi^m$, then developing the 
Fourier Jacobi expansion of a paramodular form as in (\ref{eq1}) defines a map 
$\FJ:M_k(\KN )\to \prod_{m\ge 0:\, N \vert m}J_{k,m}$.  
Now we can state a characterization. 

\begin{theorem}\label{Theorem: Main}
Let $k \in \Z_{\ge 0}$, $N \in \N$ and $\ep \in\{-1,1\}$.  
Let $\Phi = \sum_{m: N \vert m} \phi_m \xi^m \in  \prod_{m\ge 0: \,N \vert m}J_{k,m}$ 
be a formal power series whose coefficients are Jacobi forms.  
There is an $f \in M_k(\KN )^{\ep}$ such that $\Phi=\FJ(f)$ if and only if 
$\Phi$ satisfies the {\rm Involution$(\ep)$\/} condition and the growth condition of 
Definition~\ref{def3}.  
%\begin{align*}
%1)\,\,\, & %\text{\rm Growth: }
%  \forall \rho >1, \exists A >0: \forall \smtwomat{m}{r/2}{r/2}{n} \in \NXss, 
%\vert c(n,r; \phi_m) \vert \le A \rho^{n+m}. \\
%2)\,\,\, & \text{\rm Involution$(\ep)$: }
%  \forall \smtwomat{m}{r/2}{r/2}{n} \in \NXss,  
%c(n,r;\phi_m)= \ep\, c(\frac{m}{N}, -r; \phi_{nN}).
%\end{align*}
\end{theorem}
\begin{proof}
We first assume that $\Phi=\FJ(f)$ for $f \in M_k(\KN )^{\ep}$ and write 
each $T \in \NXs$ as 
$T=\smtwomat{n}{r/2}{r/2}{m}$.   For any $\rho>1$, take $\lambda>0$ with
$\rho=e^{2\pi \lambda}$.  By the Koecher principle there is an $A >0$ such that
$\vert f(\Omega) \vert \le A$ on $\{\Omega=x+iY\in\Half_2: Y > \frac{\lambda}{2}I_2 \}$.  
For $\Omega=X+i\lambda I_2$ we have the growth condition: 
\begin{align*}
\vert c(n,r; \phi_m) \vert  &=
\vert a(T;f) \vert  = 
\vert  \int_{X \in [0,1]^3} f(\Omega) e\left( -\langle \Omega, T \rangle \right) dX \vert  \\
\le &\int_{X \in [0,1]^3} \vert  f(\Omega) \vert e^{2\pi \langle \lambda I_2, T \rangle}dX \le
A \rho^{\tr(T)} = A\rho^{m+n}.
\end{align*}
For the {\rm Involution$(\ep)$\/} condition, we need to know the action of the involution $\mu_N$ on the 
Fourier expansion of $f$:
\begin{align*}
(f\vert \mu_N)(\Omega)=
\det(F_n)^{-k}  & \sum a(T;f) e\left( \langle F_N'\Omega F_N, T \rangle  \right)  \\
=  &\sum a\left( F_N T F_N'; f\right) e\left( \langle \Omega  , T \rangle  \right). 
\end{align*}
Now 
$F_N TF_N'=\frac1{\sqrt{N}} 
\smtwomat{0}{1}{-N}{0} \smtwomat{n}{r/2}{r/2}{m} \smtwomat{0}{-N}{1}{0} \frac1{\sqrt{N}} = 
\smtwomat{m/N}{-r/2}{-r/2}{Nn}$, 
so that we have the {\rm Involution$(\ep)$\/} condition:
\begin{align*}
c(n,r; \phi_m) =
 &a(\smtwomat{n}{r/2}{r/2}{m};f) =  \ep\, a(\smtwomat{n}{r/2}{r/2}{m};f\vert\mu_N) \\
 = &\ep\, a(\smtwomat{m/N}{-r/2}{-r/2}{Nn};f)
 =\ep\, c(\frac{m}{N}, -r; \phi_{Nn}).
\end{align*}

Now assume that $\Phi=\sum \phi_m \xi^m$ satisfies the growth and {\rm Involution$(\ep)$\/}  conditions.  
For any $T=\smtwomat{n}{r/2}{r/2}{m}\in \NXs$, define $a(T)$ by 
$a(T)=c(n,r;\phi_m)$.  On the set $\{\Omega=x+iY\in\Half_2: Y \ge \lambda I_2 \}$ 
the series 
$ \sum_{ T\in \NXs } a(T) e\left( \langle \Omega, T \rangle \right)$ 
is majorized by a convergent series of constants.  
To see this, choose $\rho$ with $ 1 < \rho< e^{2 \pi \lambda}$ so that by 
the growth condition there is an $A>0$ with 
$\vert a(T) \vert= \vert c(n,r; \phi_m) \vert \le A \rho^{n+m}$ and so 
\begin{align*}
\sum \vert a(T) \vert e^{-2\pi \langle Y, T \rangle} 
\le & \sum  A \rho^{m+n}e^{-2\pi \langle Y, T \rangle} 
=A \sum_T \left( \frac{\rho}{e^{2 \pi \lambda}} \right)^{m+n}  \\
\le  &A \sum_{n=0}^\infty \sum_{m=0}^\infty (2n+2m+1) \left( \frac{\rho}{e^{2 \pi \lambda}} \right)^{m+n}.  
\end{align*}
Since the convergence is uniform on compact sets, we may 
define a holomorphic function $f:\Half_2\to\C$ via 
$f(\Omega)= \sum_{ T \in \NXs }a(T) e\left( \langle \Omega, T \rangle \right)$.  

The absolute convergence of this series shows that $f\smtwomat{\tau}{z}{z}{\omega}$ is equal 
to the rearrangement $\sum_{m\in\Z_{\ge 0}: N \vert m} \phi_m(\tau,z) e(m \omega)$, or 
$f= \sum_{m\in\Z_{\ge 0}: N \vert m} {\tilde \phi}_m$.   
The invariance of $f$ under the action of the group~$\Jac$ now follows from the invariance of the 
${\tilde \phi}_m$.  In particular, we have $f \vert E_1 =f$ for 
$$
E_1=
\begin{pmatrix} 
0  & 0  & 1  &  0 \\
0 & 1  &  0  &  0  \\
-1 & 0 & 0  &  0  \\
0 & 0  &  0  & 1  
\end{pmatrix} 
\in \Jac.  
$$
Furthermore, the {\rm Involution$(\ep)$\/} condition gives us 
$$
a\left( F_N T F_N' \right) = c(\frac{m}{N}, -r; \phi_{nN})= \ep c(n,r;\phi_m)=\ep a(T), 
$$
so that 
\begin{align*}
(f\vert_k \mu_N)(\Omega) &=
\det(F_n)^{-k} \sum_{T\in \NXs} a(T) e\left( \langle F_N'\Omega F_N, T \rangle  \right) \\
=  &\sum_T a\left( F_N T F_N'\right) e\left( \langle \Omega  , T \rangle  \right)  
=\sum_T \ep a\left( T\right) e\left( \langle \Omega  , T \rangle  \right)
=\ep f(\Omega).  
\end{align*}
Following Gritsenko \cite{Grit2}, 
we have $f\vert E_1 \mu_N=f \vert \mu_N = \ep f$ and therefore that 
$f\vert (E_1 \mu_N)^2=f$.  
The group $\KN $ is generated by translations and the element 
$(E_1 \mu_N)^2=-J(N)$ so that $f \in M_k( \KN  )^{\ep}$.  
\end{proof}

\section{Aoki's method for $N=2,3$ and $4$.}
\label{sec: Aoki}

Does Theorem~\ref{Theorem: Main} remain true without the growth condition?  
A method of H. Aoki \cite{Aoki} shows that it does for $N=1$.  We successfully 
use Aoki's method to show the same for $N \le 4$.

\begin{definition} 
Let $j,k,m \in \Z$,  $N \in \N$ and $\ep \in \{-1,1\}$. Set  
\begin{align*}
\MM_k^{(j)}(N)^{\ep }   &= 
\{ \Phi  = \sum_{m\in \Z:\,m\ge Nj: \,N \vert m} \phi_m \xi^m \in  \MM_k(N)  ^{\ep} 
  \},  \\
   \ord \phi  &= \min\{ n \in \Z_{\ge 0}: \exists r \in \Z: c(n,r;\phi) \ne 0 \},\, \text{ for } \phi \in J_{k,m},\\
  J_{k,m}( j)  &= \{ \phi \in J_{k,m}: \ord \phi \ge j \}.  
  \end{align*}
\end{definition}

Here, as in Aoki  \cite{Aoki}\cite{Habuka2}, precise dimensions in specific cases follow from 
inequalities that are in general too generous.  Most dramatically, the final terms 
in the following Estimate diverge for $N > 5$ and large weights.

\begin{lemma}\label{Estimate} {\rm (Estimate)}  
Let $N \in \N$, $\ep \in \{-1,1\}$, $k \in \Z$ and set 
$\delta=\lg\left( (-1)^k \ep \right)=\begin{cases}
0, & \text{if $(-1)^k \ep=1$}\\
1, & \text{if $(-1)^k \ep=-1$.}
\end{cases}$  
We have 
\begin{align*}
&\dim M_k\left( \KN  \right)^{\ep} \le 
\dim \MM_k(N)  ^{\ep} \le 
\sum_{j=0}^{\infty} \dim \left( \MM_k^{(j)}(N)^{\ep} /  \MM_k^{(j+1)}(N)^{\ep}  \right) \le \\
 &\sum_{j=0}^{\infty} \dim J_{k,Nj}(j+\delta) \le 
\begin{cases}
\sum_{j=0}^{\infty} \sum_{i=0}^{Nj} \dim M_{k+2i-12(j+\delta)}, & \text{ $k$ even,}\\
\sum_{j=1}^{\infty} \sum_{i=1}^{Nj-1} \dim M_{k-1+2i-12(j+\delta)}, & \text{ $k$ odd.}
\end{cases}  
\end{align*}
(For $k$ odd, $N=j=1$ gives an empty second sum.)
\end{lemma}
\begin{proof}  
The first inequality follows since   
$\FJ: M_k\left( \KN  \right)^{\ep} \to \MM_k(N)  ^{\ep}$ is injective, 
the second by the filtration 
$  \MM_k^{(j)}(N)^{\ep} \supseteq  \MM_k^{(j+1)}(N)^{\ep} $.  
For the third, consider the exact sequence 
$$
0 \hookrightarrow  \MM_k^{(j+1)}(N)^{\ep} \hookrightarrow \MM_k^{(j)}(N)^{\ep} 
\to J_{k,Nj},
$$
where the final map sends $\Phi =\sum_{i = j}^{\infty} \phi_{iN}q^{iN}$ to 
$\phi_{jN}$.  The Involution$(\ep)$ condition shows that the image of the last 
map is inside $J_{k,Nj}(j+\delta)$.  This is the obvious but important point.  
If $\Phi \in \MM_k^{(j)}(N)^{\ep} $ then 
for all $ \ell < j$ we have $\phi_{N\ell}=0$, 
so that $c(\ell, r;\phi_{Nj})=\ep c(j,-r;\phi_{N\ell})=0$ and $\phi_{Nj} \in J_{k,Nj}(j)$.  
Furthermore, if $(-1)^k \ep=-1$ then 
$c(j, r;\phi_{Nj})=\ep c(j,-r;\phi_{Nj})=(-1)^k \ep c(j,r;\phi_{Nj})=-c(j,r;\phi_{Nj})$, 
so $c(j, r;\phi_{Nj})=0$ and $\phi_{Nj} \in J_{k,Nj}(j+1)$.  Thus we may uniformly write 
$\phi_{Nj} \in J_{k,Nj}(j+\delta)$. 

The last inequality follows from Lemma~3 on page 583 in Aoki \cite{Aoki},  
a consequence of the theory of differential operators in \cite{EZ}: 
$$
\dim J_{k,m}(j) \le 
\begin{cases}
 \sum_{i=0}^{m}\, \dim M_{k+2i-12j}, &\, \text{if $k$ even,}\\
 \sum_{i=1}^{m-1}\, \dim M_{k-1+2i-12j}, & \text{ if  $k$ odd, $m \ge 2$,}  \\
  0, & \text{ if  $k$ odd, $m \le 1$.}
\end{cases}  
$$
\end{proof}

\begin{lemma}\label{Summations} 
For $N \in \{1,2,3,4,5\}$ and $\ep \in \{-1,1\}$, let:
\begin{align*}
E_{N,\delta}  &=
\sum_{k  \text{ \rm even}}
\left(
\sum_{j=0}^{\infty} \sum_{i=0}^{Nj} \dim M_{k+2i-12(j+\delta)}
\right) t^k , \\
D_{N,\delta}  &=
\sum_{k  \text{ \rm odd}}
\left(
\sum_{j=1}^{\infty} \sum_{i=1}^{Nj-1} \dim M_{k+2i-12(j+\delta)}
\right) t^k .
\end{align*}
We have $E_{1,0}=\left({(1-t^4)(1-t^6)(1-t^{10})(1-t^{12})}\right)\inv$, $D_{1,1}=E_{1,1}=0$ 
and $D_{1,0}=t^{35}E_{1,0}$.  
For $2 \le N \le 5$ we have
\begin{align*}
E_{N,\delta}  &= t^{12 \delta} 
\frac{1+t^{10}+t^8+\dots +t^{14-2N}}{(1-t^4)(1-t^6)(1-t^{12})(1-t^{12-2N})}
, \\
D_{N,\delta}  &=
 t^{12 \delta} 
\frac{t^{25-2N}+t^{11}+t^9+\dots +t^{15-2N}}{(1-t^4)(1-t^6)(1-t^{12})(1-t^{12-2N})} .
\end{align*}
\end{lemma} 
\begin{proof}  
Since $\dim M_{\nu}=0$ for  $\nu<0$, 
we may make the computation slightly easier by summing over all $k \in \Z$ and using, for all $a \in \Z$, 
the identity $\sum_{k \in \Z} \dim M_{k-a} t^k =t^a/\left( (1-t^4)(1-t^6) \right)$.  
\begin{align*}
E_{N,\delta}  &=
\sum_{k  \text{ \rm even}}
\sum_{j=0}^{\infty} \sum_{i=0}^{Nj} \dim M_{k+2i-12(j+\delta)}
 t^k , \\ 
 &=
 \sum_{j=0}^{\infty} \sum_{i=0}^{Nj}
\left( \sum_{k  \text{ \rm even}}
  \dim M_{k+2i-12(j+\delta)} t^{k+2i-12(j+\delta)  } 
\right) t^{12(j+\delta)-2i  } \\
&= 
\frac{1}{ (1-t^4)(1-t^6) } 
 \sum_{j=0}^{\infty} \sum_{i=0}^{Nj}
 t^{12(j+\delta)-2i  } \\
&= 
\frac{ t^{12 \delta} }{ (1-t^4)(1-t^6) } 
 \sum_{i=0}^{\infty} \,\sum_{j=\,\luL i/N \ruL}^{\infty}
 t^{12j-2i  } \\
 &= 
\frac{ t^{12 \delta} }{ (1-t^4)(1-t^6)(1-t^{12}) } 
 \sum_{i=0}^{\infty} \,\sum_{j=\,\luL i/N \ruL}^{\infty}
\left(  t^{12j-2i  } - t^{12(j+1)-2i  } \right)\\ 
 &= 
\frac{ t^{12 \delta} }{ (1-t^4)(1-t^6)(1-t^{12}) } 
 \sum_{i=0}^{\infty}
  t^{12 \luL i/N \ruL-2i  } .  
\end{align*}
We finish by substituting $i=N\ell+\nu$ and evaluating
\begin{align*}
&\sum_{i=0}^{\infty}
t^{12 \luL i/N \ruL-2i  } =
\sum_{\nu=0}^{N-1} \sum_{\ell=0}^{\infty}  t^{12 \luL \frac{N\ell+\nu}{N} \ruL-2( N\ell+\nu) }=   \\
&\sum_{\ell=0}^{\infty}  t^{12\ell-2N\ell  } + 
\sum_{\nu=1}^{N-1} \sum_{\ell=0}^{\infty}  t^{12(\ell+1)-2N\ell-2\nu }
= \sum_{\ell=0}^{\infty}t^{ (12-2N)\ell   } 
\left( 1+ \sum_{\nu=1}^{N-1} t^{12-2\nu} \right)\\
&= \dfrac{1+t^{10}+t^8+\dots +t^{14-2N}}{1-t^{12-2N}}.  
\end{align*}

The proof for $D_{N,\delta}$ is quite similar.  
\begin{align*}
D_{N,\delta}  &=
\sum_{k  \text{ \rm odd}}
\sum_{j=1}^{\infty} \sum_{i=1}^{Nj-1} \dim M_{k-1+2i-12(j+\delta)}
 t^k , \\ 
 =&
 \sum_{j=1}^{\infty} \sum_{i=1}^{Nj-1}
\left( \sum_{k  \text{ \rm odd}}
  \dim M_{k-1+2i-12(j+\delta)} t^{k-1+2i-12(j+\delta)  } 
\right) t^{12(j+\delta)-2i +1 } \\
&= 
\frac{1}{ (1-t^4)(1-t^6) } 
 \sum_{j=1}^{\infty} \sum_{i=1}^{Nj-1}
 t^{12(j+\delta)-2i+1  } \\
&= 
\frac{ t^{12 \delta} }{ (1-t^4)(1-t^6) } 
 \sum_{i=1}^{\infty} \,\sum_{j=\,\luL( i+1)/N \ruL}^{\infty}
 t^{12j-2i+1  } \\
 = &
\frac{ t^{12 \delta} }{ (1-t^4)(1-t^6)(1-t^{12}) } 
 \sum_{i=1}^{\infty} \,\sum_{j=\,\luL (i+1)/N \ruL}^{\infty}
\left(  t^{12j-2i+1  } - t^{12(j+1)-2i +1 } \right)\\ 
 &= 
\frac{ t^{12 \delta} }{ (1-t^4)(1-t^6)(1-t^{12}) } 
 \sum_{i=1}^{\infty}
  t^{12 \luL (i+1)/N \ruL-2i +1 } .  
\end{align*}

We finish by substituting $i=N\ell+\nu$ and evaluating
\begin{align*}
&\sum_{i=1}^{\infty}
t^{12 \luL (i+1)/N \ruL-2i+1  } =
\sum_{\nu=1}^{N} \sum_{\ell=0}^{\infty}  t^{12 \luL \frac{N\ell+\nu+1}{N} \ruL-2( N\ell+\nu)+1 }=   \\
&\sum_{\nu=1}^{N-1} \sum_{\ell=0}^{\infty}  t^{12(\ell+1)-2N\ell-2\nu +1}
+\sum_{\ell=0}^{\infty}  t^{12(\ell+2)-2(N\ell +N) +1} \\
= &\sum_{\ell=0}^{\infty}t^{ (12-2N)\ell   } 
\left( \sum_{\nu=1}^{N-1} t^{13-2\nu}+ t^{25-2N} \right)\\
&= \dfrac{t^{11}+t^9+\dots +t^{15-2N}+ t^{25-2N}}{1-t^{12-2N}}.  
\end{align*}
The proof for the case $N=1$ is similar and is given in Aoki \cite{Aoki}. 
\end{proof}

\begin{corollary}
For $N \in \{2,3,4\}$ and $\ep \in \{-1,1\}$ or for $N=1$ and $\ep=1$, 
all the inequalities in the Estimate of Lemma~\ref{Estimate} are equalities.  
\begin{align*}
\forall k \in \Z, \quad &\FJ: M_k\left( \KN  \right)^{\ep} \to \MM_k(N)  ^{\ep}
 \text{ is an isomorphism.}  \\
 \forall k \text{ \rm even,} \quad
 &\dim J_{k, Nj}(j+\delta) = \sum_{i=0}^{Nj} \dim M_{k+2i-12(j+\delta)}, \\
  \forall k \text{ \rm odd,} \quad
 &\dim J_{k, Nj}(j+\delta) = \sum_{i=1}^{Nj-1} \dim M_{k-1+2i-12(j+\delta)},
\end{align*}
\begin{align*}
&\sum_{k=0}^{\infty} 
\dim M_k\left( \KN   \right)^{+} t^k= E_{N,0}+D_{N,1} = \\
 &\dfrac{1+ t^{10}+t^8+\dots+t^{14-2N}+t^{12}\left(t^{11}+t^9+\dots +t^{15-2N}+ t^{25-2N}\right)}{(1-t^4)(1-t^6)(1-t^{12})(1-t^{12-2N})} , \\
 &\sum_{k=0}^{\infty} 
\dim M_k\left( \KN   \right)^{-} t^k= E_{N,1}+D_{N,0} = \\
 &\dfrac{t^{12}\left(1+ t^{10}+t^8+\dots+t^{14-2N}\right)+t^{11}+t^9+\dots +t^{15-2N}+ t^{25-2N}}{(1-t^4)(1-t^6)(1-t^{12})(1-t^{12-2N})} . 
\end{align*}
\end{corollary}
\begin{proof}
Rewriting the inequalities of Lemma~\ref{Estimate} as 
$\dim M_k\left( \KN   \right)^{+} \le \operatorname{coeff}( E_{N,0}+D_{N,1} ,t^k) $ and as 
$\dim M_k\left( \KN   \right)^{-}\le \operatorname{coeff}( E_{N,1}+D_{N,0} ,t^k) $, we have %a bound 
$
\dim M_k\left( \KN   \right) =\dim M_k\left( \KN   \right)^{+} + \dim M_k\left( \KN   \right)^{-} 
\le  \operatorname{coeff}( E_{N,0}+D_{N,1}+ E_{N,1}+D_{N,0} ,t^k) 
$.  
If we can show equality here, 
we have $\dim M_k\left( \KN   \right)^{+}=\operatorname{coeff}( E_{N,0}+D_{N,1} ,t^k) $ 
and  $\dim M_k\left( \KN   \right)^{-}=\operatorname{coeff}( E_{N,1}+D_{N,0} ,t^k) $
and the proof is complete.  
However, 
the generating functions 
$\sum_{k \in \Z} \dim M_k\left( \KN   \right) t^k$ are known for $N=2,3$ and $4$ 
and one checks equality with $E_{N,0}+E_{N,1}+D_{N,0}+D_{N,1}$.  
\end{proof}

\section{The Generating Function of $K(4)$}
\label{sec: K4}

For any natural number~$t$, the paramodular group $K(t^2)$ is conjugate, 
by an element of $\SptwoQ$, to the following group $\Gt$, which is a subgroup 
of $\Gamma_2$ containing the principal subgroup $\Gamma_2(t)$;  
$$
\Gt =
\begin{pmatrix} 
*  &  t*  &  *   &  t* \\
t*& *  &  t*  &  *  \\
*& t*  & *  &  t*  \\
t*& *  &  t*  &  *  
\end{pmatrix} 
\cap \SptwoZ, 
\text{ where $*\in \Z$.}
$$

%\centerline{Table 1.  Cycle types and generating functions for $S_3 \times S_3 \subseteq S_6.$} 
\begin{alignat*}{5}
& \text{Table 1.} &  &{}  &  & \text{$S_6$ cycles.}  &  &{} &   &{} \\
& \underline{S_3 \times S_3} &  &{}  &  &  \underline{M\in S_6}  &  &{} &   &\,\, \quad  \underline{g(M;t)} \\
& (1) \times (1)\quad  & &1\qquad &  & (1) \quad &  &1 \qquad  &     &\frac{(1+t^5)(1-t^8)}{(1-t^2)^5}  \\
& (12) \times (1)\quad  & &3\quad &  & {} \quad &  &{} \quad  &  & {}  \\
& (1) \times (12)\quad  & &3\quad &  & (12) \quad &  &6 \quad  &     &\frac{(1-t^5)(1-t^8)}{(1-t^2)^2(1+t^2)^3}  \\
& (12) \times (12)\quad  & &9\quad &  & (12)(34) \quad &  &9 \quad  &     &\frac{(1+t^5)(1-t^8)}{(1-t^2)^3(1+t^2)^2}  \\
& (123) \times (1)\quad  & &2\quad &  & {} \quad &  &{} \quad  &  & {}  \\
& (1) \times (123)\quad  & &2\quad &  & (123) \quad &  &4 \quad  &     &\frac{(1+t^5)(1-t^8)}{(1-t^2)(1+t^2+t^4)^2}  \\
& (123) \times (12)\quad  & &6\quad &  & {} \quad &  &{} \quad  &  & {}  \\
& (12) \times (123)\quad  & &6\quad &  & (12)(345)\quad &  &12 \,\,  &     &\frac{(1-t^5)(1-t^8)}{(1{+}t^2)(1{-}t^2{+}t^4)(1{+}t^2{+}t^4)}  \\
& (123) \times (123)\quad  & &4\quad &  & (123)(456)\quad &  &4 \quad  &     &\frac{(1+t^5)(1-t^8)}{(1+t^2)^3(1+t^2+t^4)}  \\
\end{alignat*}

The proof is that  
$\diag(1,t,1,t\inv) K(t^2) \diag(1, t\inv,1 , t) =\Gt$.  
In Igusa \cite{Igusa}, we may find the generating function for the character~$X_k$ 
of the representation of $\SptwoF  \simeq \Gamma_2/\Gamma_2(2)$ 
acting on $M_k(\Gamma_2(2))$.  Since $\Gttwo$ contains the principal subgroup $\Gamma_2(2)$, 
Igusa's result allows us to calculate the generating function for $\Gttwo$ by the formula 
$$
\sum_{k=0}^{\infty} \dim M_k(\Gttwo)\,t^k = 
\frac{1}{| G |} \sum_{M \in G} \, \sum_{k=0}^{\infty} X_k(M) t^k,   
$$
where $G=\Gttwo/\Gamma_2(2)$ is a finite group.  
Now $\SptwoF$ is isomorphic to the symmetric group $S_6$ via the 
permutation of the six odd theta characteristics and 
the group $G \simeq \SLtwoF \times \SLtwoF$  corresponds to a choice of 
$S_3 \times S_3 \subseteq S_6$  by the action of $\SLtwoF$ on the three 
even theta characteristics.  We separate the elements $M\in G$ into conjugacy 
classes, which may be given by cycle types inside $S_6$, and give Igusa's computation 
(page 401, \cite{Igusa}) of  $g(M;t)=\sum_{k=0}^{\infty} X_k(M) t^k$ for these conjugacy classes.  
Table~1 lists the cycle types in both $S_3 \times S_3$ and $S_6$ and gives the number
of elements that have that cycle type.  
%For each cycle type $M \in S_6$, the generating function 
%$g(M;t)$ is given.  

This gives  
$\sum_{k=0}^{\infty} \dim M_k(K(4))\,t^k = \sum_{k=0}^{\infty} \dim M_k(\tilde \Gamma(2))\,t^k = $ 
\begin{align*}
\frac{1}{36} & \{  g((1);t)+6g((12);t)+9g((12)(34);t)+   \\
& 4g((123);t)+12g((12)(345);t)+4g((123)(456);t) \}   \\
&\qquad = \frac{(1+t^{12})(1+t^6+t^7+t^{8}+t^9+t^{10}+t^{11}+t^{17}) }{(1-t^4)^2(1-t^6)(1-t^{12})}.  
\end{align*}

%Added by Ibukiyama: 
%By the way, we can show that $\dim J_{k,4j}^{cusp}(j)=\max\{\dim J_{k,4j}(j)-1,0\}$
%by comparing the Taylor expansion and the Theta expansion of Jacobi forms 
%as in Eichler-Zagier \cite{EZ}.
%By this, we can give a generating function for $\dim S_{k}(\Gamma[4])$ also.  
%We can give it also by calculating the dimension of the image of generalized $\Phi$-operator on 
%the boundary and of course these coincide. This gives you a good cross check.

%Added by Cris: 
We mention a good cross check now that we know  $\dim M_{k}(K(4))$.  
We can show that $\dim J_{k,4j}^{\rm cusp}(j)=\max\{\dim J_{k,4j}(j)-1,0\}$
by comparing the Taylor expansion and the theta expansion of Jacobi forms 
as in Eichler-Zagier \cite{EZ}.  
By this, we can also give  upper bounds for  $\dim S_{k}(K(4))$.  
These upper bounds coincide with the true dimension of $ S_{k}(K(4))$ 
computed from the known  dimension of $ M_{k}(K(4))$ and 
the dimension of the image of generalized $\Phi$-operator on 
the boundary.

\section{An Example: $\pi_{12}\circ\FJ: S^4(K(31))\to \prod_{j=1}^{12} J_{4,31j}^{\text{cusp}}$}
\label{sec: Example}

Although the linear relations from the Involution$(\ep)$ condition are 
practical to implement on a computer, the growth condition is not.  
It is natural to wonder about the effect of omitting the growth condition and we work out one 
example with this in mind.  In light of \thmref{Theorem: Main}, when we compute formal 
power series over Jacobi forms satisfying the involution condition, either there will only be 
Fourier Jacobi expansions of paramodular forms or there will also be solutions with 
rapidly growing coefficients.  We consider the subspaces $S_4(K(31))^{-}$ and 
$$\SSS=\{ f\in S_4(K(31) )^{+}: \ord_{\xi} \FJ(f) \ge 62\}$$ for the following reasons:  
The dimensions $\dim J_{k,m}^{\text{cusp}}$ are known for $k \ge 2$, 
see \cite{EZ}\cite{Skor}, and so we only need to generate sufficiently many linearly independent 
elements of $J_{k,m}^{\text{cusp}}$ to compute inside this space.  
Especially in weight four, see \cite{GritHulek}, theta blocks are a convenient way to construct 
Jacobi forms.  For $d \in \N^8$ with $d\cdot d=2N$, we have 
$T(d)(\tau, z)= \prod_{i=1}^8 \vartheta(\tau, d_iz)\in J_{4,N}$; 
here $\vartheta(\tau, z)=\sum_{n\in \Z} (-1)^n q^{\frac{(2n+1)^2}{8}} \zeta^{\frac{2n+1}{2}}$.  
It is easy to see that $T(d)$ is a cusp form if $d$ has both even and odd entries.  
We select $ K(p)$ for prime level $p$ because T. Ibukiyama \cite{Ibukp}\cite{Ibuk} has given 
$\dim S_k( K(p) )$ for $k \ge 3$; this information allows us to measure our computations 
against a known dimension.  For weight~$4$, we have 
\begin{align*}
&\dim S_4\left(  K(p) \right) =  \\
\frac{p^2}{576}+\frac{p}{8}-\frac{143}{576} +  &
\left(\frac{p}{96}-\frac{1}{8}\right)\left(\dfrac{-1}{p}\right)
+\frac{1}{8}\left(\dfrac{2}{p}\right)
+\frac{1}{12}\left(\dfrac{3}{p}\right)
+\frac{p}{36}\left(\dfrac{-3}{p}\right)
\end{align*}
$$
\dim J_{4,m}^{\text{cusp}}  = 
\sum_{j=1}^m \left(  \{\{ 4+2j\}\}- \ldL ({j^2}/{4m}) \rdL  \right),   
$$
where we  let $\ldL x \rdL = \max\{ m \in \Z: m \le x \}$ be the greatest integer function and 
where $\{\{ k \}\}=\dim S_k(\operatorname{SL}_2(\Z))$.   

V. Gritsenko  has a lifting 
$\Grit:J_{k,N}^{\text{cusp}}\to S_k(\KN )^{\ep}$ for $\ep=(-1)^k$ with the property that 
the Fourier Jacobi expansion of $\Grit(\phi)$ has leading term $\phi \,\xi^N$, see \cite{Grit1}.    
In selecting a generic example,  we avoid these lifts because  their Fourier coefficients 
satisfy special linear relations.  The first prime~$p$ for which the map 
$\Grit:J_{4,p}^{\text{cusp}}\to S_4( K(p) )$ does not surject is $p=31$; here 
$\Grit\left(  J_{4,31}^{\text{cusp}}\right)$ is five dimensional and $S_4( K(31) )$ six.  
By subtracting off the Gritsenko lift of the leading Fourier Jacobi coefficient 
we have $S_4( K(31) )^{+}=\Grit\left(  J_{4,31}^{\text{cusp}}\right) \oplus \SSS$.  
We will compute $12$ coefficients of the Fourier Jacobi expansions 
from $S_4( K(31) )$ 
 in accordance 
with the following Lemma, 
noting here that  $\frac{k}{10}\frac{p^2+1}{p+1}=\frac{4}{10}\frac{31^2+1}{31+1} = 12.025$.  

\begin{lemma}\label{Lemma: Little}
Let $p$ be a prime, $J,M,k \in \N$ and $\ep \in \{-1,1\}$. 
Let $\pi_M: \prod_{j=1}^{\infty} J_{k, pj} \to \prod_{j=1}^{ \ldL M/p \rdL} J_{k, pj} $ be projection.  
The map $\pi_{pJ}\circ \FJ:  S_k( K(p) )^{\ep} \to \prod_{j=1}^{J} J_{k, pj}^{\text{\rm cusp}}$ 
injects for $J \ge \ldL \dfrac{k}{10} \left( \dfrac{p^2+1}{p+1} \right) \rdL$.  
\end{lemma}
\begin{proof}
For $T \in \pX=\{ \smtwomat{a}{b}{b}{c} > 0: a,2b,c \in \Z \text{ and } p \vert c \}$, define the Minimum function~$m$ via 
$m(T)= \min_x T[x]$ over $x \in \Z^2\setminus \{0\}$.  
It is known that the $T \in \pX$ with $m(T) \le \frac{k}{10}\frac{p^2+1}{p+1}$ 
are a determining set of Fourier coefficients for $S_k( K(p) )^{\ep}$, see \cite{Para}.  
Consider $ f\in S_k( K(p) )^{\ep}$ such that 
\begin{equation}
\label{eq2}  
\forall T=\smtwomat{n}{r/2}{r/2}{m}\in\pX: \frac{m}{p} \le \frac{k}{10}\frac{p^2+1}{p+1}, \,
a(T;f)=0.  
\end{equation}
We need to show that such $f$ vanish.  
Take any $T\in\pX$ satisfying $m(T) \le \frac{k}{10}\frac{p^2+1}{p+1}$.  
By reduction we have 
$T=\smtwomat{ a }{ b }{ b }{ c } 
\left[ \smtwomat{ \alpha }{ \beta }{\gamma }{ \delta  } \right]$ 
for some $ \smtwomat{ \alpha }{ \beta }{\gamma }{ \delta  } \in \GLtwoZ$ 
and $ 0 \le 2b \le c \le a$; 
in this case $c = m(T)$.  
If $ p \vert \beta$ then 
$ \smtwomat{ \alpha }{ \beta }{\gamma }{ \delta  }  \in {\hat \Gamma}^0(p)$ 
and $a(T)= \pm a\left( \smtwomat{ a }{ b }{ b }{ c } \right)=0$ by  (\ref{eq2}) since 
$\frac{c}{p} \le c \le  \frac{k}{10}\frac{p^2+1}{p+1}$. 
If $\beta$ is prime to $p$, let $r \in \Z$ solve $\beta r \equiv \delta \mod p$; 
then $\s=\smtwomat{ \alpha }{ \beta }{\gamma }{ \delta  }\inv 
\smtwomat{ 0 }{ 1 }{ 1 }{ r } \in  {\hat \Gamma}^0(p)$  and we have 
$T[\s]= \smtwomat{ a }{ b }{ b }{ c }
\left[\smtwomat{ 0 }{ 1 }{ 1 }{ r } \right]=
\begin{pmatrix} c & b+rc \\ b+rc & cr^2+2br+a \end{pmatrix} \in \pX$ 
so that $p \vert ({cr^2+2br+a })$.  
In this case 
$$
a(T)=\det(\s)^k a(T[\s]) = 
\ep \det(\s)^k 
a\left(\begin{pmatrix} \frac{cr^2+2br+a}{p} & -(cr+b) \\ -(cr+b) & pc \end{pmatrix}\right)=0
$$
by (\ref{eq2}) because 
$\frac{pc}{p}=c \le  \frac{k}{10}\frac{p^2+1}{p+1}$. 
Since $a(T)=0$ for all $T$ with $m(T) \le \frac{k}{10}\frac{p^2+1}{p+1}$, 
we have $f=0$.  
\end{proof}

For $p=31$ and $k=4$, the following Proposition computes the first $J=12$ 
Jacobi form coefficients of any formal power series that satisfies the Involution$(\ep)$ condition 
and finds that they are all initial Fourier-Jacobi expansions of 
paramodular cusp forms.  
This makes it at least plausible that the involution condition alone 
characterizes the Fourier Jacobi expansions from $S_4( K(31) )^{\ep}$ from 
among all formal power series over Jacobi forms.  
And that is the point of this computation--- to show that the growth condition 
may be superfluous.  

\begin{proposition}\label{Proposition: Example}
Let $k, p, J \in \N$ with $p$ prime.  
Define the subspaces
\begin{align*}
A(J) &=\{ \Phi= \sum_{j=1}^J \phi_{jp} \,\xi^{jp} \in \prod_{j=1}^{J} J_{k, jp}^{\text{\rm cusp}}: 
\text{ $\Phi$ satisfies {\rm Involution$(-)$}}  \} \text{ and } \\
B(J) &=\{ \Phi= \sum_{j=2}^J \phi_{jp} \,\xi^{jp} \in \prod_{j=1}^{J} J_{k, jp}^{\text{\rm cusp}}: 
\text{ $\Phi$ satisfies {\rm Involution$(+)$}}  \} .  
\end{align*}
For $k=4$ and $p=31$, 
the subspace $A(12)$ is trivial and the subspace $B(12)$ 
is one dimensional and is spanned by 
$\Phi_0=\psi_{62}\xi^{62}+\psi_{93}\xi^{93}+\dots +\psi_{12\cdot 31}\xi^{12\cdot 31}$
with initial expansions
\begin{align*}
&\psi_{62} = {\mathbf q^2}( 
-\zeta^{22}+7 \zeta^{21}-15 \zeta^{20}-3 \zeta^{19}+50 \zeta^{18}-37 \zeta^{17}-47 \zeta^{16}  \\
&+19 \zeta^{15}+74 \zeta^{14} 
+49 \zeta^{13}-163 \zeta^{12}-{13} \zeta^{11}{+}67 \zeta^{10}{+}28 \zeta^9+108 \zeta^8  \\
&-84 \zeta^7-106 \zeta^6-74 \zeta^5+114 \zeta^4
+162 \zeta^3-84 \zeta^2-54 \zeta+6-54/\zeta  \\ 
&-84/\zeta^2 
+162/\zeta^3+114/\zeta^4-74/\zeta^5-106/\zeta^6 
-84/\zeta^7+108/\zeta^8 \\ 
&+28/\zeta^9+67/\zeta^{10} 
-13/\zeta^{11}-163/\zeta^{12}+49/\zeta^{13}+74/\zeta^{14} 
+19/\zeta^{15} \\
&-47/\zeta^{16}-37/\zeta^{17}+50/\zeta^{18}-3/\zeta^{19}-15/\zeta^{20}+7/\zeta^{21}-1/\zeta^{22})\\
%&  \\
\,\, +&{\mathbf q^3}(
\zeta^{27}-5\zeta^{26}+5\zeta^{25}+11\zeta^{24}-19\zeta^{23}-2\zeta^{22}-5\zeta^{21}+21\zeta^{20}+39\zeta^{19} \\ 
&-47\zeta^{18}  
-5\zeta^{17}-64\zeta^{16}+19\zeta^{15}+133\zeta^{14}-25\zeta^{13}+17\zeta^{12}-131\zeta^{11} \\
&-52\zeta^{10}+71\zeta^{9}-3\zeta^{8}  
+159\zeta^{7}-37\zeta^{6}-49\zeta^{5}{-}38\zeta^{4}{-}86\zeta^{3}+10\zeta^{2} \\
&+26\zeta+112
+26/\zeta+10/\zeta^{2}-86/\zeta^{3}  
-38/\zeta^{4}-49/\zeta^{5}-37/\zeta^{6} \\
&+159/\zeta^{7}-3/\zeta^{8}+71/\zeta^{9}-52/\zeta^{10}-131/\zeta^{11}+17/\zeta^{12}  
-25/\zeta^{13} \\
&+133/\zeta^{14}+19/\zeta^{15}-64/\zeta^{16}-5/\zeta^{17}-47/\zeta^{18}+39/\zeta^{19}+21/\zeta^{20}  \\
&-5/\zeta^{21}-2/\zeta^{22}-19/\zeta^{23}{+}11/\zeta^{24}{+}5/\zeta^{25}{-}5/\zeta^{26}{+}1/\zeta^{27})
+ O({\mathbf q^4});  \\
&\psi_{93} = {\mathbf q^2}( \operatorname{coeff}(\psi_{62},q^3) )
+ O({\mathbf q^3}) .
\end{align*}
\end{proposition}
\begin{proof}
It is convenient to denote
$J_{k,m}^{\text{cusp}}(\nu)= 
\{ \phi \in J_{k,m}^{\text{cusp}}: \ord  \phi \ge \nu \}$.  
Let $\Phi=0\cdot \xi^{31}+\phi_{62}\xi^{62}+\phi_{93}\xi^{93}+\dots +\phi_{ 31\nu}\xi^{ 31\nu}\in B(\nu)$.  
The space $J_{4,62}^{\text{cusp}}$ is spanned by the $9$ theta blocks $T(d)$ for $d=$ 
[1, 1, 1, 1, 2, 4, 6, 8], 
[1, 1, 1, 2, 2, 2, 3, 10], 
[1, 1, 1, 2, 2, 4, 4, 9], 
[1, 1, 1, 2, 3, 6, 6, 6], 
[1, 1, 2, 2, 2, 2, 5, 9], 
[1, 1, 2, 4, 4, 5, 5, 6], 
[1, 2, 2, 2, 2, 3, 7, 7], 
[1, 3, 4, 4, 4, 4, 5, 5], 
[2, 2, 2, 2, 3, 3, 3, 9].  
The Involution$(+)$ condition tells us that for all 
$\smtwomat{n}{r/2}{r/2}{m}\in \Xtwo(31)$ we have 
$c(n,r;\phi_m)=c(\frac{m}{31},-r;\phi_{31n})$.   
Setting $n=1$ and $m=62$ in condition~Involution$(+)$, we have 
$$
c(1,r;\phi_{62})=c(2,-r;\phi_{31})=c(2,-r;0)=0, 
$$
so that the $q^1$-coefficients of $\phi_{62}$ vanish.  
The subspace $J_{4,62}^{\text{cusp}}(2)$ is spanned by one element, $\psi_{62}$, 
which is the following linear combination of the above nine theta blocks: 
 $\psi_{62}= (-3,-5,-1,-2,-1,0,1,0,1) \cdot (T(d_1),\dots, T(d_{9}))$.  The initial expansion 
 of $\psi_{62}$ is as given above.  Thus $\phi_{62}$ is some multiple of $\psi_{62}$, 
 say $\phi_{62}=\alpha \psi_{62}$ for $\alpha \in \C$, and the subspace $B(2)$ is at most 
 one dimensional.  
 
The space $J_{4,93}^{\text{cusp}}$ is spanned by the $16$ theta blocks $T(c)$ for $c=$ 
[1, 1, 1, 1, 1, 1, 6, 12], 
[1, 1, 1, 1, 1, 6, 8, 9], 
[1, 1, 1, 1, 2, 3, 5, 12], 
[1, 1, 1, 1, 2, 4, 9, 9], 
[1, 1, 1, 1, 4, 6, 7, 9], 
[1, 1, 1, 3, 5, 6, 7, 8], 
[1, 1, 2, 2, 2, 6, 6, 10], 
[1, 1, 2, 3, 3, 3, 3, 12], 
[1, 1, 2, 3, 3, 4, 5, 11], 
[1, 1, 2, 6, 6, 6, 6, 6], 
[1, 2, 3, 3, 3, 3, 8, 9], 
[1, 3, 4, 4, 6, 6, 6, 6], 
[2, 2, 2, 2, 2, 2, 9, 9], 
[2, 2, 2, 2, 2, 3, 6, 11], 
[2, 3, 3, 3, 3, 3, 4, 11], 
[3, 4, 5, 5, 5, 5, 5, 6].  
For $n=1$ and $m=93$ the Involution$(+)$ conditions are 
$c(1,r;\phi_{93})=c(3,-r;\phi_{31}) =0$ so that $\phi_{93}\in J_{4,93}^{\text{cusp}}(2)$.  
The subspace $J_{4,93}^{\text{cusp}}(3)$ is trivial and the subspace $J_{4,93}^{\text{cusp}}(2)$ 
is spanned by the following four linear combinations of theta blocks:
\newcommand\mm{\text{-}}
\begin{align*}
Q_1  &= (\mm1,\mm1,\mm6,\mm6,\mm4,\mm1,0,\mm1,2,0,1,0,0,0,0,0)\cdot(T(c_1),\dots,T(c_{16})), \\
Q_2  &=(\mm2,\mm1,\mm9,\mm6,\mm3,0,\mm1,\mm1,3,0,0,0,1,0,0,0  )\cdot(T(c_1),\dots,T(c_{16})), \\ 
Q_3  &= (\mm1,\mm1,\mm4,\mm2,0,0,1,1,\mm2,0,0,0,0,1,0,0  )\cdot(T(c_1),\dots,T(c_{16})), \\ 
Q_4  &= (1,0,1,3,2,\mm1,0,\mm1,\mm5,0,0,0,0,0,1,0  ) \cdot(T(c_1),\dots,T(c_{16})). 
\end{align*}
%## QQ1     -4 t5 - t1 - t8 - 6 t4 + 2 t9 - t6 - t2 - 6 t3 + t11
 %## QQ2      -t8 - t7 - 6 t4 - 3 t5 - t2 + 3 t9 - 2 t1 + t13 - 9 t3
 %## QQ3       t8 + t7 - 2 t4 + t14 - t2 - 2 t9 - t1 - 4 t3
% ## QQ4      -t8 + 3 t4 - 5 t9 + t1 + 2 t5 + t15 + t3 - t6
Some Fourier coefficients for these $Q_i$ are in Table~2.  
We use the Involution$(+)$ condition for $n=2$ and $m=93$ to find the 
$q^2$-coefficients of $\phi_{93}$.  
\begin{equation}
\label{eq3} 
c(2,r;\phi_{93})= c(3, -r; \phi_{62})= \alpha c(3, -r; \psi_{62})
\end{equation}
The coefficients $c(3, -r; \psi_{62})$ are known and displayed in the statement of 
the Proposition.  The unique element $\phi_{93} \in J_{4,93}^{\text{cusp}}(2)$  
satisfying equation~(\ref{eq3}) is $\alpha \psi_{93}$ where $\psi_{93}= -Q_1-Q_4$.  
This shows that the subspace $B(3)$ is at 
most one dimensional.  Continuing in this way on a computer, we showed that 
$J_{4,31j}^{\text{cusp}}(j)=\{0\}$ for $j=3,\dots,12$ and hence that 
$\dim B(12) \le 1$.

We discuss the minus space.  
The space $J_{4,31}^{\text{cusp}}$ is spanned by $5$ theta blocks $T(b)$ for $b=$ 
[1, 1, 1, 1, 1, 2, 2, 7], [1, 1, 1, 1, 1, 4, 4, 5], [1, 1, 1, 1, 2, 2, 5, 5], [1, 1, 2, 2, 2, 4, 4, 4],  [2, 2, 3, 3, 3, 3, 3, 3].  
For even weights, the Involution$(-)$ conditions are quite restrictive.  
We have $c(j,r;\phi_{31j})=- c(j, -r; \phi_{31j})$, so that the $q^j$-coefficients of $\phi_{31j}$ 
must vanish.  However, the $q^1$-coefficients of the five theta blocks $T(b_i)$ are 
already linearly independent, so $A(1)$ is trivial.  
Now that we know that the first Jacobi coefficient vanishes, 
by the same reasoning as for the plus space, 
the only possible element of $A(2)$ is a multiple of $\psi_{62}$; however the extra 
condition that the $q^2$-coefficients of $\phi_{62}$ vanish shows that $A(2)$ is 
trivial.  The triviality of $A(12) $ now follows from $J_{4,31j}^{\text{cusp}}(j)=\{0\}$ for $j=3,\dots,12$.  

%[1, 1, 1, 1, 1, 2, 2, 7], [1, 1, 1, 1, 1, 4, 4, 5], [1, 1, 1, 1, 2, 2, 5, 5], [1, 1, 2, 2, 2, 4, 4, 4],  [2, 2, 3, 3, 3, 3, 3, 3]
%  These 5 theta blocks span J_{4,31}.  
 
 By \thmref{Theorem: Main} we have a map 
 $\pi_{12\cdot 31} \circ \FJ: S_4( K(31) )^{-} \to A(12)$ and, 
 by \lemref{Lemma: Little}, this map is injective; 
 hence $S_4( K(31))^{-} $ is trivial.  
 From Ibukiyama's result, $\dim S_4( K(31) )=6$, 
 we may conclude that $\dim S_4( K(31))^{+}=6$ and $\dim \SSS=1$.  
 Therefore $\Phi_0 \in \pi_{12\cdot 31} \FJ(\SSS) \subseteq B(12)$ 
 and $\dim B(12) =1$.  
 
\end{proof}

From another point of view, 
the merit of the preceding computations consists in providing 
upper bounds for the dimension  of spaces of paramodular cusp forms.  
In this particular case, 
relying on Ibukiyama's dimension formula for the existence of forms, 
we have shown the following Corollary.  
\begin{corollary}
\label{dilly}
$\dim S_4( K(31) )^{+}=6$ and $\dim S_4( K(31) )^{-}=0$.  
\end{corollary}

\section{Final Remarks}
\label{sec: Remarks}

We conclude by comparing the Involution condition with the following weaker inequality;  
 for general~$N$, we cannot even show that the right hand side is finite:  
\begin{equation}
\label{Hmm}
\dim S_k\left( \KN \right)^{ (-1)^k} \le 
\sum_{j=1}^{\infty} \dim J_{k,Nj}^{\text{\rm cusp}}(j).
\end{equation}

\newpage
\centerline{Table 3}
\centerline{Values of $\dim J_{k,5j}^{\text{\rm cusp}}(j)$}
\smallskip
\begintable
\   ${}_k\backslash {}^j$ \ &\  $1$   \ &\    $2$   \ &\   $3$  \ &\   $4$  \ &\    $5$     \ &\    $\sum_{j=1}^{\infty} \dim J_{k,5j}^{\text{\rm cusp}}(j)$   \ &\   $\dim S_k(K(5))^{ (-1)^k }$ \crthick
$      1$ &  $    $  & $  $ &    $ $ & $ $  \ &\    $ $        \ &\    $0$   \ &\   $0$ \nr
$      2$ &  $    $  & $  $ &    $ $ & $ $  \ &\    $ $       \ &\    $0$   \ &\   $0$ \nr
$      3$ &  $   $  & $  $ &    $ $ & $ $   \ &\    $ $    \ &\    $0$   \ &\   $0$ \nr
$      4$ &  $   $  & $  $ &    $ $ & $ $   \ &\    $ $    \ &\    $0$   \ &\   $0$ \nr
$     5$ &  $  1$  & $  $ &    $ $ & $ $  \ &\    $ $    \ &\    $1$   \ &\   $1$  \cr
$     6$ &  $  1$  & $  $ &    $ $ & $ $  \ &\    $ $    \ &\    $1$   \ &\   $1$  \nr
$     7$ &  $  1$  & $ {}$ &    $ $ & $ $   \ &\    $ $     \ &\    $1$   \ &\   $1$ \nr
$     8$ &  $  2$  & $ 0$ &    $ $ & $ $   \ &\    $ $    \ &\    $2$   \ &\   $2$ \nr
$      9$ &  $  2$  & $ 0$ &    ${}$ & $ $   \ &\    $ $     \ &\    $2$   \ &\   $2$ \nr
$      10$ &  $  3$  & $ 1$ &    $ 0$ & $ $   \ &\    $ $     \ &\    $4$   \ &\   $4$ \cr
$      11$ &  $  3$  & $ 1$ &    $ 0$ & $ $  \ &\    $ $     \ &\    $4$   \ &\   $4$ \nr
$      12$ &  $  4$  & $ 2$ &    $ 1$ & $ $  \ &\    $ $     \ &\    $7$   \ &\   $6$ \nr
$      13$ &  $  3$  & $ 2$ &    $ 1$ & $0$   \ &\    $ $     \ &\    $6$   \ &\   $5$ \nr
$      14$ &  $  5$  & $ 3$ &    $ 2$ & $0$   \ &\    $ $      \ &\    $10$   \ &\   $8$ \nr
$      15$ &  $  4$  & $ 3$ &    $ 2$ & $ 0$   \ &\    $ $      \ &\    $9$   \ &\   $8$ 
\endtable

For the case $N=31$ and $k=4$ we have demonstrated the equality 
$\dim S_4\left(K(31) \right)^{+} =
\sum_{j=1}^{20} \dim J_{4,31j}^{\text{\rm cusp}}(j)$ or 
$6=5+1+0+0+\dots+0$.  
However tempting it may be to replace the $20$ by $\infty$, 
we cannot be sure about that equality because we have only computed 
$ \dim J_{4,31j}^{\text{\rm cusp}}(j)=0$ for $3 \le j \le 20$.  We can, however, 
be certain about inequalities;  for example with $N=29$ and $k=4$, we can show the inequality 
$\dim S_4\left(K(29) \right)^{+} <
\sum_{j=1}^{\infty} \dim J_{4,29j}^{\text{\rm cusp}}(j)$ or 
$5<5+1+0+\dots$.   The space $ J_{4,58}^{\text{\rm cusp}}(2)$ is one dimensional, 
spanned by $\Psi$ say, but there is no element in $ J_{4,87}^{\text{\rm cusp}}(2)$ 
whose $q^2$-terms equal the $q^3$-terms of $\Psi$.  Hence there does exists a 
$\Phi= \Psi \xi^2$ satisfying the Involution$(+)$ conditions to second order that is not the 
initial Fourier Jacobi expansion of any paramodular cusp form from $S_4\left(K(29) \right)^{+} $.  
All the $\Phi$ that satisfy the Involution$(+)$ condition to third order, however, are the 
initial Fourier Jacobi expansions of  paramodular cusp forms from $S_4\left(K(29) \right)^{+} $.  
Hence, again in this example, the Involution condition continues to compute the space 
$S_4\left(K(29) \right)^{+} $ correctly even when the inequality~(\ref{Hmm}) is strict.  

One case where the convergence of the series 
$\sum_{j=1}^{\infty} \dim J_{k,Nj}^{\text{\rm cusp}}(j)$ is known for all weights~$k$ is $N=5$.  
Here the sum need only be taken to $j=\ldL k/2 \rdL$ because 
$\ord \phi \le (k+2m)/12$ for $\phi \in J_{k,m}$.  A more refined estimate of 
Gritsenko and Hulek \cite{GritHulek} shows that $j \le \ldL (3k-6)/8 \rdL$ suffices when $N=5$.  
Since $N=5$ is also the first level where the inequality~(\ref{Hmm}) can be strict, 
it is of some interest to ponder this data.  
Table~3 gives the values of $\dim J_{k,5j}^{\text{\rm cusp}}(j)$ for $1 \le k \le 15$ and for 
$j \le \ldL (3k-6)/8 \rdL$.  These were computed by using theta blocks to span the spaces of Jacobi forms.  
For weights $k \le 15$, the dimensions of $S_k(K(5))^{\pm}$ in Table~3 may be found in a manner similar to 
that used to prove Corollary~{\ref{dilly}}.    
One sees in Table~3 that the inequality~(\ref{Hmm}) is already strict for weight~$k=12$ and hence that the 
method that was used for $1 \le N \le 4$ to  prove the surjectivity of 
$\FJ: M_k\left( \KN  \right)^{\ep} \to \MM_k(N)  ^{\ep}$ for all weights 
will not work for $N=5$.

\newpage
\centerline{Table 2}
\smallskip
\begintable
\   $r$ \ &\  $c(2,r;Q_1)$   \ &\    $c(2,r;Q_2)$   \ &\   $c(2,r;Q_3)$  \ &\   $c(2,r;Q_4)$\crthick
$      0$ &  $   114$  & $ 300$ &    $6$ & $-226$ \nr
$      1$ &  $  -38$  & $ -145$ &    $-69$ & $12$ \nr
$     2$ &  $  14$  & $ -47$ &    $89$ & $-24$\nr
$     3$ &  $  -60$  & $ -1$ &    $41$ & $146$ \nr
$      4$ &  $  -34$  & $84$ &    $-72$ & $72$ \cr
$      5$ &  $  40$  & $-69$ &    $-53$ & $9$\nr
$      6$ &  $  65$  & $ -27$ &    $41$ & $-28$ \nr
$     7$ &  $  15$  & $209$ &    $74$ & $-174$ \nr
$      8$ &  $  -42$  & $ -113$ &    $0$ & $45$ \nr
$     9$ &  $  -49$  & $ -137$ &    $-103$ & $-22$ \cr
$      10$ &  $  -65$  & $-72$ &    $-49$ & $117$ \nr
$      11$ &  $  137$  & $ 303$ &    $190$ & $-6$ \nr
$      12$ &  $  -33$  & $-93$ &    $-55$ & $16$ \nr
$     13$ &  $  61$  & $ -44$ &    $-42$ & $-36$ \nr
$      14$ &  $  -79$  & $ -10$ &    $-72$ & $-54$ 
\endtable

\newpage
\centerline{Table 2 (continued)}
\smallskip
\begintable
\   $r$ \ &\  $c(2,r;Q_1)$   \ &\    $c(2,r;Q_2)$   \ &\   $c(2,r;Q_3)$  \ &\   $c(2,r;Q_4)$\crthick
$      15$ &  $  -42$  & $0$ &    $67$ & $23$ \nr
$      16$ &  $  67$  & $ 30$ &    $101$ & $-3$ \nr
$     17$ &  $  -40$  & $-99$ &    $-122$ & $45$ \nr
$      18$ &  $  73$  & $149$ &    $19$ & $-26$ \nr
$     19$ &  $  -57$  & $-55$ &    $-3$ & $18$ \cr
$      20$ &  $  3$  & $ -23$ &    $31$ & $-24$ \nr
$     21$ &  $  7$  & $-6$ &    $-5$ & $-2$ \nr
$     22$ &  $  -7$  & $9$ &    $-28$ & $9$ \nr
$      23$ &  $  19$  & $ 30$ &    $24$ & $0$ \nr
$      24$ &  $  -18$  & $ -35$ &    $-8$ & $7$ \cr
$     25$ &  $  7$  & $14$ &    $1$ & $-12$ \nr
$     26$ &  $  -1$  & $-2$ &    $0$ & $6$ \nr
$      27$ &  $  0$  & $ 0$ &    $0$ & $-1$
\endtable

\end{document}

\end{document}
\end{document}